\documentclass[oneside]{amsart}
\usepackage{amsfonts, amsmath, amsthm, amssymb, relsize, upgreek, enumerate, mathrsfs, latexsym}
\usepackage [english]{babel}
\usepackage{graphicx}
\usepackage{cjhebrew}
\usepackage{tikz-cd}
\usepackage{comment}
\usepackage [autostyle, english = american]{csquotes}
\MakeOuterQuote{"}
\usepackage{harpoon}

\newtheorem{thm}{Theorem}[section]
\newtheorem{prop}[thm]{Proposition}

\newtheorem{lem}[thm]{Lemma}

\newtheorem{clm}[thm]{Claim}
\newtheorem*{thmA}{Theorem A}
\newtheorem*{thmB}{Theorem B}
\newtheorem*{thmC}{Theorem C}
\theoremstyle{definition}
\newtheorem{defn}[thm]{Definition}
\newtheorem*{ack}{Acknowledgements}

\begin{document}

\author{Christian Schulz}
\address
{Department of Mathematics\\University of Illinois at Urbana-Champaign\\1409 West Green Street\\Urbana, IL 61801}

\title{Undefinability of multiplication in Presburger arithmetic with sets of powers}

\begin{abstract}
    We begin by proving that any Presburger-definable image of one or more sets of powers has zero natural density. Then, by adapting the proof of a dichotomy result on o-minimal structures by Friedman and Miller, we produce a similar dichotomy for expansions of Presburger arithmetic on the integers. Combining these two results, we obtain that the expansion of the ordered group of integers by any number of sets of powers does not define multiplication.
\end{abstract}

\maketitle

\section{Introduction}

    In this paper, we consider expansions of \textit{Presburger arithmetic,} the structure $\mathcal{Z} = (\mathbb{Z}, <, +)$,\footnote{Typical treatments of Presburger arithmetic instead focus on $\mathcal{N} = (\mathbb{N}, +)$. In this paper, it is more convenient to use $\mathcal{Z}$. The change makes no difference for tameness; note that $\mathcal{N}$ defines a subset $Z$ of $\mathbb{N}^2$ and relations thereon that form an isomorphic copy of $\mathcal{Z}$ along with the inclusion map $\mathbb{N} \hookrightarrow Z$; and likewise the subset $\mathbb{N}$ is definable in $\mathcal{Z}$.} by sets of the form $a^\mathbb{N} = \{a^i : i \in \mathbb{N}\}$, where $a \in \mathbb{Z}_{>1}$.
    
    Expansions of $\mathcal{Z}$ occupy a great many possibilities in terms of tameness. On the one hand, $\mathcal{Z}$ itself is well-understood to be quite tame: it admits effective quantifier elimination upon expansion by the (definable) relations $\equiv_{n}$ of equivalence modulo $n$ for each integer $n > 1$ and thus has a decidable theory, and additionally it satisfies model-theoretic tameness conditions like NIP. On the other hand, the expansion of $\mathcal{Z}$ by the multiplication function, i.e. the ring of integers,\footnote{In the integer ring, $<$ is definable via Lagrange's four-square theorem.} is sufficiently complex to perform G\"odel coding and hence defines any arithmetical set. In a previous paper, Hieronymi and Schulz \cite{HS} show that $(\mathbb{Z}, <, +, a_1^\mathbb{N}, a_2^\mathbb{N})$ does not have a decidable theory provided that $a_1$ and $a_2$ are \textit{multiplicatively independent,} i.e. their powers are all distinct except in the trivial case $a_1^0 = a_2^0$. Here we will show that these structures however do not define the multiplication function:
    
    \begin{thmA}
        Let $a_1, \dots, a_n \in \mathbb{Z}_{>1}$. Then $(\mathbb{Z}, <, +, a_1^\mathbb{N}, \dots, a_n^\mathbb{N})$ does not define multiplication.
    \end{thmA}
    
    This has been an open question since at latest 1996; it is part of Question 9 in \cite{MV}. The context behind the consideration of this particular class of structures deserves a great deal of explanation. We note that without the ordering relation, $(\mathbb{Z}, +, a_1^\mathbb{N}, \dots, a_n^\mathbb{N})$ is superstable of $U$-rank $\omega$, as shown by Conant \cite{Conant}; hence, without $<$ there can be no definition of multiplication. Conant's methods, however, do not readily transfer to the case where ordering is included. The introduction of $<$ instead places us within the realm of \textit{$k$-automatic} expansions of $\mathcal{Z}$, which usually do not satisfy NIP or NTP$_2$ but may still enjoy decidable theories.
    
    A $k$-automatic set (of arity $d \in \mathbb{N}$) is a subset of $\mathbb{Z}^d$ whose base-$k$ representations form a regular language.\footnote{To account for our convention of using $\mathcal{Z}$ as mentioned in the previous footnote, we elaborate: a set $S \subseteq \mathbb{Z}^d$ is $k$-automatic if $\phi(S) \cap \mathbb{N}^d$ is $k$-automatic, where $\phi$ is each of the $2^d$ maps that flip the signs of any subset of the coordinates.} It is shown by Bruy\`ere et al. \cite{BHMV} that a set is $k$-automatic if and only if it is definable in $\mathcal{Z}_k = (\mathbb{Z}, <, +, V_k)$, where $V_k : \mathbb{Z} \to k^\mathbb{N}$ maps $n$ to the largest power of $k$ dividing $n$, or $0$ if $n = 0$. The expansion $\mathcal{Z}_k$ is known as \textit{B\"uchi arithmetic.} Moreover, it is shown by B\`es \cite{Bes97} that if $S \subseteq \mathbb{Z}^d$ is a $k$-automatic set not definable in $\mathcal{Z}$, then $(\mathbb{Z}, <, +, S)$ defines the set $k^\mathbb{N}$. This set is not definable in $\mathcal{Z}$,\footnote{The set $k^\mathbb{N}$ is not semilinear and thus cannot be Presburger-definable by Ginsburg and Spanier \cite{GinsSpanier}.} so these results taken together establish the status of $V_k$ as a ``maximal'' $k$-automatic set and $k^\mathbb{N}$ as a ``minimal'' (non-Presburger-definable) $k$-automatic set.
    
    By B\"uchi's theorem, the theory of $\mathcal{Z}_k$ is decidable. Hence, so is the theory of any expansion of $\mathcal{Z}$ by $k$-automatic sets for a fixed $k$. Let $a_1, a_2 \in \mathbb{Z}_{>1}$ be multiplicatively independent. A natural next step is then to examine the properties of $(\mathbb{Z}, <, +, S_{a_1}, S_{a_2})$ where $S_{a_i}$ for $i = 1, 2$ is an $a_i$-automatic set not definable in $\mathcal{Z}$. The first major result in this direction is the Cobham-Sem\"enov theorem, which states that this expansion is always nontrivial, i.e. it may never be the case that $S_{a_2}$ is definable in $\mathcal{Z}_{a_1}$. This was proven by Cobham \cite{cobham} for the case where $S_{a_1}, S_{a_2} \subseteq \mathbb{N}$ and extended to the more general case $S_{a_1}, S_{a_2} \subseteq \mathbb{N}^d$ by Sem\"enov \cite{Semenov}.
    
    It is also known that the theory of such a structure is never decidable. For the case $S_{a_1} = V_{a_1}, S_{a_2} = V_{a_2}$ this was proven by Villemaire \cite{Villemaire}, and the more general case where $S_{a_1} = V_{a_1}$ was shown by B\`es \cite{Bes97}. Both of these authors also show that in the respective cases, the resulting structure defines the multiplication function on $\mathbb{N}$ and hence cannot have a decidable theory. However, in \cite{HS}, multiplication is not shown to be definable in $(\mathbb{Z}, <, +, S_{a_1}, S_{a_2})$. Theorem A of this paper thus demonstrates that a departure from Villemaire and B\`es' method is required in the general case by exhibiting for the first time a structure, namely $(\mathbb{Z}, <, +, a_1^\mathbb{N}, a_2^\mathbb{N})$, that defines non-Presburger $a_1$-automatic and $a_2$-automatic sets and has an undecidable theory but does not define the multiplication on $\mathbb{Z}$. One may note that this leaves mostly open the question of intermediate expansions, i.e. of characterizing for precisely which $S_{a_1}, S_{a_2}$ the structure $(\mathbb{Z}, <, +, S_{a_1}, S_{a_2})$ defines multiplication, which we will leave for future work.
    
    Theorem A follows from the conjunction of two results, one number-theoretic and one model-theoretic. The first is an application of a proposition of Schlickewei and Schmidt \cite{SS}. Its full statement is quite technical, so we will defer discussion of the exact proposition until section 3, but in essence it provides a bound for solutions to a particular form of equation involving a finitely generated multiplicative subgroup of a number field that lie outside one of a finite number of hyperplanes. We are able to leverage this result by noting that $a_1^\mathbb{N} \cup \dots \cup a_n^\mathbb{N}$ is a subset of such a subgroup, namely the subgroup whose generators are $\{a_1, \dots, a_n\}$. Using this, we derive:
    
    \begin{thmB}
        Let $E = \bigcup_i a_i^\mathbb{N}$, where $a_i \in \mathbb{Z}_{>1}$ for $1 \leq i \leq n$, and let $f : \mathbb{Z}^M \to \mathbb{Z}$ be a Presburger-definable function. Then $f(E^M)$ has zero natural density.
    \end{thmB}
    
    In order to utilize this result, we build on previous work by Friedman and Miller \cite{FM}. There, they prove that when we add a set $E$ to the signature of an o-minimal expansion $\mathcal{R}$ of the ordered real additive group, if the image of $E$ under any function definable in $\mathcal{R}$ has a small closure, then every set definable in the resulting structure either has interior or is small. Here ``small'' can refer to any of several sparseness conditions on subsets of the real line: nowhere dense, Lebesgue-null, etc. In this paper, we produce a corresponding result for expansions of the ordered group of integers. Here, by $(\mathbb{Z}, <, +, E)^\#$, we mean $(\mathbb{Z}, <, +, (S))$ where $S$ ranges over all subsets of powers of $E$.
    
    \begin{thmC}
        Let $\mathcal{I}$ be any set-theoretic ideal on $\mathbb{Z}$. Let $E \subseteq \mathbb{Z}$ be such that, for every $M \in \mathbb{N}$ and $h : \mathbb{Z}^M \to \mathbb{Z}$ Presburger-definable, the image $h(E^M)$ is in $\mathcal{I}$. Then every subset of $\mathbb{Z}$ definable in $(\mathbb{Z}, <, +, E)^\#$ either contains arbitrarily long pieces of some arithmetic progression or lies in $\mathcal{I}$.
    \end{thmC}
    
    This is a significant result in its own right, and it merits some elaboration. Here ``$S \subseteq \mathbb{Z}$ contains arbitrarily long pieces of some arithmetic progression'' means that there exists $N \in \mathbb{Z}^+$ such that for any $\ell \in \mathbb{Z}^+$, there exist $\ell$ elements of $S$ with a constant difference of $N$. This is a simultaneous strengthening of two more common largeness conditions used in number-theoretic literature. In particular, a set $S$ satisfying the above contains arbitrarily long arithmetic progressions, and it is piecewise syndetic. We additionally note that this condition is equivalent to stating that for some $S' \subseteq S$, there exists a unary affine function $f$ such that $f(S')$ is a thick set (hence, this condition is also a weakening of thickness).
    
    On the other hand, note that we need not assume much about the particular sparseness condition we wish to use; any set-theoretic ideal on $\mathbb{Z}$ will do. We need only assume that subsets and finite unions of sparse sets are sparse. The obvious choice is to use zero natural density, but there are other useful notions of sparseness that give rise to the same dichotomy. The reason we are able to give such a general statement instead of listing specific sparseness conditions as Friedman and Miller did has to do with the fact that in the discrete setting, we have no reason to take the topological closure of the image of $E$.
    
    Theorem C has broader applications within the model theory of expansions of Presburger arithmetic, but within this paper its application will be in proving Theorem A. We will thus state here the proof of Theorem A using Theorems B and C, with the rest of this paper devoted to proving Theorems B and C.
    
    \begin{proof}[Proof of Theorem A]
        Let $E = a_1^\mathbb{N} \cup \dots \cup a_n^\mathbb{N}$. By Theorem B, the image of $E^M$ under any Presburger function $h$ has zero natural density. Let $\mathcal{I}$ be the collection of sets of zero natural density; this is a set-theoretic ideal. Thus, we may apply Theorem C and conclude that every subset of $\mathbb{Z}$ definable in $(\mathbb{Z}, <, +, E)^\#$ either contains arbitrarily long pieces of some arithmetic progression or has zero natural density. Note that $(\mathbb{Z}, <, +, E)^\#$ defines the sets $a_i^\mathbb{N}$ for $1 \leq i \leq n$. Then it suffices to show that $(\mathbb{Z}, <, +, \cdot)$ defines some set that does not contain arbitrarily long pieces of any arithmetic progression yet has positive natural density.
        
        Let $S \subseteq \mathbb{Z}$ contain arbitrarily long pieces of the same arithmetic progression; then as mentioned previously, $S$ is piecewise syndetic. So it suffices to give an example of a subset of $\mathbb{Z}$ (or $\mathbb{N}$) that is arithmetical and has positive natural density but is not piecewise syndetic. The squarefree integers form such a set: it is not piecewise syndetic (folklore), its natural density is well-known to be $\frac{6}{\pi^2}$, and it is defined by the formula $\neg \exists y, z : 1 < y \wedge y \cdot y \cdot z = x$.
    \end{proof}
    
    \begin{ack}
        The author would like to thank Philipp Hieronymi and Alexi Block Gorman for guidance, as well as Alf Dolich and Chris Miller, who had some initial work in this direction but allowed the author to work out the details and publish the results. The author was partially supported by National Science Foundation Grant no. DMS -- 1654725.
    \end{ack}
    
    \section{Preliminaries}
    
    Throughout, $\mathbb{N}$ denotes the natural numbers, which by our convention comprise all nonnegative integers including zero. ``Defines'' refers to definability by a first-order formula \textit{without} parameters, although note that this makes little difference in our scope; in the structures we will consider, any constant is definable. We denote the cardinality of set $A$ by $|A|$.
    
    One of our main results uses the concept of the natural density of a subset of $\mathbb{Z}$. Natural density is typically defined for subsets of $\mathbb{N}$, so it is worth explaining our definition:
    
    \begin{defn}
	    The \textit{upper natural density} of a subset $A \subseteq \mathbb{Z}$ is the limit superior:
	    
	    $$\limsup_{n \to \infty} \frac{|A \cap \{-n, -n+1, \dots, n-1, n\}|}{2n + 1}.$$
	    
	    The \textit{lower natural density} of $A$ is the limit inferior:
	    
	    $$\liminf_{n \to \infty} \frac{|A \cap \{-n, -n+1, \dots, n-1, n\}|}{2n + 1}.$$
	    
	    If these two expressions are equal, we refer to both as simply the \textit{natural density} of $A$, denoted $d(A)$.
	\end{defn}
	
	Note that for $A \subseteq \mathbb{N}$, the natural density of $A$ as a subset of $\mathbb{N}$ is \textit{not} equal to its natural density as a subset of $\mathbb{Z}$; rather, it is equal to the natural density of $A \cup -A$ as a subset of $\mathbb{Z}$. This choice of convention is so that the natural density on $\mathbb{Z}$ may act as a finitely additive probability measure. We thus check that the appropriate properties are satisfied, as well as some other convenient properties:

	\begin{lem}
	\label{density_props}
	    Let $A, B, C \subseteq \mathbb{Z}$. Then all of the following hold:
	    \begin{enumerate}[(i)]
	        \item $d(\varnothing) = 0$.
	        \item $d(\mathbb{Z}) = 1$.
	        \item Let $A, B$ be disjoint, with $C = A \cup B$. If any two of $A, B, C$ have a natural density, so does the remaining set, and $d(A) + d(B) = d(C)$.
	        \item If $d(A) = 0$ and $B \subseteq A$, then $d(B) = 0$.
	        \item If $d(A)$ is well-defined, then for all $k \in \mathbb{Z}$, $d(A + k) = d(A)$, where $A + k$ is a translation of the set $A$.
	        \item If $d(A)$ is well-defined, then for all $0 \neq k \in \mathbb{Z}$, $d(kA) = \frac{d(A)}{k}$, where $kA$ is a scaled version of the set $A$.
	        \item Any Presburger-definable subset of $\mathbb{Z}$ has a natural density.
	        \item For $k > 1$, $d(k^\mathbb{N}) = 0$.
	    \end{enumerate}
	\end{lem}
	\begin{proof}
	    These results are elementary, and proofs are not qualitatively different from the corresponding proofs for natural density over $\mathbb{N}$. We will however provide complete proofs here for the convenience of the reader.
	
	    \begin{enumerate}[(i)]
	        \item $$\lim_{n \to \infty} \frac{|\varnothing \cap \{-n, -n+1, \dots, n-1, n\}|}{2n + 1} = \lim_{n \to \infty} \frac{0}{2n + 1} = 0.$$
	        \item $$\lim_{n \to \infty} \frac{|\mathbb{Z} \cap \{-n, -n+1, \dots, n-1, n\}|}{2n + 1} = \lim_{n \to \infty} \frac{2n + 1}{2n + 1} = 1.$$
	        \item Assume the natural densities of $A$ and $B$ are well-defined; then:
	        
	        \begin{align*}
	            & \lim_{n \to \infty} \frac{|C \cap \{-n, -n+1, \dots, n-1, n\}|}{2n + 1} \\
	            =& \lim_{n \to \infty} \frac{|(A \cup B) \cap \{-n, -n+1, \dots, n-1, n\}|}{2n + 1} \\
	            \intertext{Because the union $A \cup B$ is disjoint:}
	            =& \lim_{n \to \infty} \frac{|A \cap \{-n, -n+1, \dots, n-1, n\}| + |B \cap \{-n, -n+1, \dots, n-1, n\}|}{2n + 1} \\
	            =& \lim_{n \to \infty} \left( \frac{|A \cap \{-n, -n+1, \dots, n-1, n\}|}{2n + 1} + \frac{|B \cap \{-n, -n+1, \dots, n-1, n\}|}{2n + 1} \right) \\
	            =& \; d(A) + d(B).
	        \end{align*}

	        \noindent Similarly assume the natural densities of $A$ and $C$ are well-defined; then:
	        
	        \begin{align*}
	            & \lim_{n \to \infty} \frac{|B \cap \{-n, -n+1, \dots, n-1, n\}|}{2n + 1} \\
	            =& \lim_{n \to \infty} \frac{|(C \setminus A) \cap \{-n, -n+1, \dots, n-1, n\}|}{2n + 1} \\
	            =& \lim_{n \to \infty} \frac{|C \cap \{-n, -n+1, \dots, n-1, n\}| - |A \cap \{-n, -n+1, \dots, n-1, n\}|}{2n + 1} \\
	            =& \lim_{n \to \infty} \left( \frac{|C \cap \{-n, -n+1, \dots, n-1, n\}|}{2n + 1} - \frac{|A \cap \{-n, -n+1, \dots, n-1, n\}|}{2n + 1} \right) \\
	            =& \; d(C) - d(A).
	        \end{align*}

	        \noindent The case where the natural densities of $B$ and $C$ are well-defined is analogous.
	        
	        \item Natural density is nonnegative, as it is a limit of a ratio of nonnegative quantities; so by (iii), it suffices to show that the natural density of $B$ is well-defined. For all $n$,
	        
	        $$0 \leq \frac{|B \cap \{-n, -n+1, \dots, n-1, n\}|}{2n + 1} \leq \frac{|A \cap \{-n, -n+1, \dots, n-1, n\}|}{2n + 1},$$
	        
	        so this follows from the squeeze theorem.
	        
	        \item Note that $|(A+k) \cap \{-n, \dots, n\}| = |A \cap \{-n-k, \dots, n-k\}|$ and that:
	        
	        $$|(A \cap \{-n-k, \dots, n-k\}) \: \triangle \: (A \cap \{-n, \dots, n\})| \leq 2k.$$
	        
	        But $\lim_{n \to \infty} \frac{2k}{2n+1} = 0$, so again by the squeeze theorem, we have $d(A+k) = d(A)$.
	        
	        \item
	        
	        \begin{align*}
	            & \lim_{n \to \infty} \frac{|kA \cap \{-n, -n+1, \dots, n-1, n\}|}{2n + 1} \\
	            =& \lim_{n \to \infty} \frac{|kA \cap \{-nk, -nk+1, \dots, nk-1, nk\}|}{2nk + 1} \\
	            =& \lim_{n \to \infty} \frac{|A \cap \{-n, -n+1, \dots, n-1, n\}|}{2nk + 1} \\
	            =& \lim_{n \to \infty} \frac{|A \cap \{-n, -n+1, \dots, n-1, n\}|}{2nk + k} \lim_{n \to \infty}  \frac{2nk + k}{2nk + 1} \\
	            =& \; \frac{d(A)}{k} \cdot 1.
	        \end{align*}

	        \item Any Presburger-definable set $A \subseteq \mathbb{Z}$ is a finite union of sets of the form $\{an + b : n \in \mathbb{N}\}$ for integers $a, b$. So by (iii) it suffices to assume $A = \{an + b : n \in \mathbb{N}\}$. Then by (v) it suffices to let $b = 0$, and by (vi) it suffices to let $a = 1$; i.e. $A = \mathbb{N}$. Then:
	        
	        $$\lim_{n \to \infty} \frac{|\mathbb{N} \cap \{-n, -n+1, \dots, n-1, n\}|}{2n + 1} = \lim_{n \to \infty} \frac{n + 1}{2n + 1} = \frac{1}{2}.$$
	        
	        \item
	        
	        \begin{align*}
	            & \lim_{n \to \infty} \frac{|k^\mathbb{N} \cap \{-n, -n+1, \dots, n-1, n\}|}{2n + 1} \\
	            =& \lim_{n \to \infty} \frac{|k^\mathbb{N} \cap \{-k^n, -k^n+1, \dots, k^n-1, k^n\}|}{2k^n + 1} \\
	            =& \lim_{n \to \infty} \frac{n+1}{2k^n + 1} = 0.
	        \end{align*}
	    \end{enumerate}
	\end{proof}
	
	We will make key use of a result of Schlickewei and Schmidt \cite{SS}, so we use several of their notational conventions in algebraic number theory. Let $\ast$ denote the elementwise product of two tuples, i.e. $(x_1, \dots, x_n) \ast (y_1, \dots, y_n) = (x_1 y_1, \dots, x_n y_n)$. Let $V(F)$ be the set of places of a number field $F$. For each place $v$, we let $|\cdot|_v$ be the associated absolute value and let $F_v$ be the completion of $F$ at $v$. We let $||r||_v = |r|_v^{[F_v : \mathbb{Q}_p] / [F : \mathbb{Q}]}$ where $p$ is the place of $\mathbb{Q}$ that $v$ divides.
	
	We then define the \textit{absolute multiplicative height} $H(x)$:
	
	\begin{defn}
	    For a number field $F$ and a tuple $x = (x_1, \dots, x_n) \in F^n$, the \textit{absolute multiplicative height} $H(x)$ is the product of $\max \{1, ||x_1||_v, \dots, ||x_n||_v\}$ over all places $v$ of $F$. The \textit{logarithmic height} $h(x)$ is equal to $\log H(x)$.
	\end{defn}
	
	As we now show, the function $H$ has a much simpler characterization if we assume that $x$ is rational:
	
	\begin{lem}
	\label{H_fractions}
		Let $x = (\frac{a_1}{b}, \dots, \frac{a_n}{b}) \in \mathbb{Q}^n$ be a tuple of rational numbers such that $a_1, \dots, a_n, b \in \mathbb{Z}$, with $b > 0$ and $\gcd(a_1, \dots, a_n, b) = 1$. Then $H(x) = \max \{|a_1|, \dots, |a_n|, b\}$.
	\end{lem}
	\begin{proof}
		Note that every element of $\mathbb{Q}^n$ has precisely one such representation, which may be obtained by finding a common denominator for the components and then dividing out common divisors if any remain. So $x \mapsto \max \{|a_1|, \dots, |a_n|, b\}$ is indeed a function $\mathbb{Q}^n \to \mathbb{R}^+$.
		
		The set of places of $\mathbb{Q}$ is the set $V(\mathbb{Q}) = \{2, 3, \dots, \infty\}$ of primes and infinity. So for such a place $p$ we have $||r||_p = |r|_p^{[\mathbb{Q}_p : \mathbb{Q}_p] / [\mathbb{Q} : \mathbb{Q}]} = |r|_p$. Then:
		
		\begin{align*}
		H(x) =& \prod_{p \in M(\mathbb{Q})} \max \{1, |x_1|_p, \dots, |x_n|_p\} \\
		=& \prod_{p \in M(\mathbb{Q})} \max \left\{1, \left|\frac{a_1}{b}\right|_p, \dots, \left|\frac{a_n}{b}\right|_p\right\}.
		\end{align*}
		
		Now if $p$ is a prime, then $\left|\frac{a_i}{b}\right|_p = \left|\frac{1}{b}\right|_p$ for some $i$; otherwise we would have $p \mid a_1, \dots, a_n, b$, contradicting the greatest common divisor assumption. So $\max \left\{1, \left|\frac{a_1}{b}\right|_p, \dots, \left|\frac{a_n}{b}\right|_p\right\} = \left|\frac{1}{b}\right|_p = p^k$, where $k$ is the exponent on $p$ in the factorization of $b$. Then the product of this expression over all primes $p$ is the product of each maximal prime power dividing $b$, which is just $b$. Thus:
		
		\begin{align*}
		H(x) &= b \cdot \max \left\{1, \left|\frac{a_1}{b}\right|_\infty, \dots, \left|\frac{a_n}{b}\right|_\infty\right\} \\
		&= b \cdot \frac{1}{b} \max \{b, |a_1|, \dots, |a_n|\} \\
		&= \max \{|a_1|, \dots, |a_n|, b\}.
		\end{align*}
	\end{proof}

	\section{A Sparseness Result on Sets of Powers}
	
	We begin by recalling (viii) in Lemma \ref{density_props}, namely that for any $k \in \mathbb{Z}_{>1}$, we have $d(k^\mathbb{N}) = 0$. Moreover, (v), (vi), and (iii) give us that if we shift, scale, and/or take the union of multiple such sets, respectively, the result will still have zero natural density. The goal of this section is to extend this result to general Presburger-definable images of unions of sets of powers.
	
	Note that this property is nontrivial; in fact, a set may be very sparse indeed yet have a Presburger-definable image that is all of $\mathbb{Z}$. One example: consider the set $S$ containing, for each $i \in \mathbb{N}$, the numbers $10^{10^i}$ and $10^{10^i} + i$. Clearly $S$ has zero natural density, and hence the observations of the previous paragraph also apply to $S$. But the Minkowski difference $S - S$, i.e. the image of $S^2$ under the function $(x, y) \mapsto x - y$, contains every integer and hence has natural density $1$. So to state that any Presburger-definable image of a set has zero density is a strictly stronger sparseness condition. The remainder of this section will be focused on proving that unions of sets of powers satisfy this sparseness condition.
	
	We leverage the following:
	
	\begin{prop}[Proposition A of \cite{SS}]
	\label{schlickewei_schmidt}
		Let $F$ be a number field of degree $d$. Let $\Gamma \subset (F^*)^n$ be a finitely generated subgroup of rank $r$. Then the set of all points $y = x \ast z \in F^n$ with $x \in \Gamma$, $z \in (\mathbb{Q}^*)^n$, $y \cdot (1, \dots, 1) = 1$, and $h(z) \leq \frac{1}{4n^2} h(x)$ is contained in the union of not more than $2^{30n^2} (32n^2)^r d^{3r+2n}$ proper linear subspaces of $F^n$.
	\end{prop}
	
	Using this proposition, we will prove the following number-theoretic result:
	
	\begin{lem}
	\label{diophantine_density}
		Let $k_1, \dots, k_n$ be nonzero rational numbers, and let $a_1, \dots, a_n \in \mathbb{Z}_{>1}$. Then the set $C$ of integers $c$ for which there exist $e_1, \dots, e_n \in \mathbb{N}$ satisfying:
		
		\begin{equation}
		\label{defining_equation}
			k_1 a_1^{e_1} + \dots + k_n a_n^{e_n} = c
		\end{equation}
		
		has zero natural density.
	\end{lem}
	\begin{proof}
	    We will prove this claim via induction; first, assume $n = 1$. Then $C$ is simply the set $\{k_1 a_1^e : e \in \mathbb{Z}, e \geq 0\}$, which clearly has zero natural density. For the remainder of this proof we then assume $n > 1$ and assume the lemma is true with $n$ replaced by any smaller positive integer.
	    
	    We let $\Gamma \subseteq (\mathbb{Q^*})^n$ be the subgroup of elements of $(\mathbb{Q^*})^n$ where each component is a power of the corresponding $a_i$; the rank of $\Gamma$ is $n$. We let $F = \mathbb{Q}$. Under these conditions, the statement of Proposition \ref{schlickewei_schmidt} becomes: the set of all points $y = x \ast z \in \mathbb{Q}^n$ with $x \in \Gamma$, $z \in \mathbb{Q}^n$, $y \cdot (1, \dots, 1) = 1$, and $h(z) \leq \frac{1}{4n^2} h(x)$ is contained in the union of not more than $2^{30n^2} (32n^2)^n$ proper linear subspaces of $\mathbb{Q}^n$. We will also note that the condition $h(z) \leq \frac{1}{4n^2} h(x)$ is equivalent to $H(z)^{4n^2} \leq H(x)$.
	    
	    Let $S$ be the set of all tuples $(x, c)$ such that $x$ is of the form $(a_1^{e_1}, \dots, a_n^{e_n})$ for $e_1, \dots, e_n \in \mathbb{N}$ and such that (\ref{defining_equation}) holds. For $S' \subseteq S$ we write $\pi(S')$ for the projection of $S'$ onto the last coordinate, which is a subset of $\mathbb{Z}$; in particular $\pi(S) = C$. We will cover $S$ by subsets and show that the projection of each subset has upper natural density zero.
	    
	    In particular, let $S_1$ be the set of $(x, c) \in S$ such that $H((\frac{k_1}{c}, \dots, \frac{k_n}{c}))^{4n^2} \leq H(x)$, and let $S_2 = S \setminus S_1$. Proposition \ref{schlickewei_schmidt} then gives us that there exist finitely many proper linear subspaces of $\mathbb{Q}^n$ such that for each $(x, c) \in S_1$,  $x \ast (\frac{k_1}{c}, \dots, \frac{k_n}{c})$ lies in one of these subspaces. For each such subspace $P$ we let $S_P$ denote the subset containing $(x, c) \in S_1$ such that $x \ast (\frac{k_1}{c}, \dots, \frac{k_n}{c}) \in P$.
	    
	    Fix some $P$. Then there exists a nontrivial linear dependence that all points in $P$ satisfy; i.e. there exists a nonzero vector $d = (d_1, \dots, d_n)$ such that for all $p \in P$ we have $d \cdot p = 0$. For each $(x, c) \in S_P$ we thus have that:
	    
	    \begin{equation}
		\label{linear_dependence}
	        d_1 \frac{k_1}{c} a_1^{e_1} + \dots + d_n \frac{k_n}{c} a_n^{e_n} = 0.
	    \end{equation}
	    
	    Without loss of generality (reordering coordinates), we now assume $d_n \neq 0$. We rearrange (\ref{linear_dependence}) and isolate $k_n a_n^{e_n}$:
	    
	    \begin{equation}
		\label{linear_dependence_solved}
	        k_n a_n^{e_n} = -\frac{d_1 k_1}{d_n} a_1^{e_1} - \dots - \frac{d_{n-1} k_{n-1}}{d_n} a_{n-1}^{e_{n-1}}.
	    \end{equation}
	    
	    We may then plug (\ref{linear_dependence_solved}) into (\ref{defining_equation}) and combine:
	    
	    \begin{equation}
		\label{defining_after_linear}
	        \left(k_1 - \frac{d_1 k_1}{d_n}\right) a_1^{e_1} + \dots + \left(k_{n-1} - \frac{d_{n-1} k_{n-1}}{d_n}\right) a_{n-1}^{e_{n-1}} = c.
	    \end{equation}
	    
	    Letting $k'_i = k_i - \frac{d_i k_i}{d_n}$ gives us:
	    
	    \begin{equation}
		\label{defining_n_minus_1}
	        k'_1 a_1^{e_1} + \dots + k'_{n-1} a_{n-1}^{e_{n-1}} = c.
	    \end{equation}
	    
	    If any $k'_i$ equals zero, that term may be disregarded; this equation then either simplifies to $0 = c$ (in which case there are no solutions) or else satisfies the conditions of the lemma for fewer terms than $n$, allowing us to apply the inductive hypothesis and conclude that $\pi(S_P)$ has zero natural density. So because $\pi(S_1)$ is a finite union of $\pi(S_P)$, $\pi(S_1)$ also has zero natural density.
	    
	    We now consider $S_2$ and let $h \in \mathbb{Z}^+$; assume $h > \max a_i$. Let $b$ be the smallest positive integer such that $k_1 b, \dots, k_n b$ are all integers, let $m = \max \{|k_1|, \dots, |k_n|\}$, and let $S_h \subseteq S_2$ contain only those $(x, c)$ where $h \geq |c| > m$.
	    
	    Let $(x, c) \in S_h$. Then $\frac{|k_i|}{c}$ is at most $1$ for each $i$, so by Lemma \ref{H_fractions}, $H((\frac{k_1}{c}, \dots, \frac{k_n}{c}))$ is the lowest common denominator of the $\frac{k_i}{c}$. Because $\frac{bk_i }{bc}$ is a ratio of two integers, we must have that $H((\frac{k_1}{c}, \dots, \frac{k_n}{c})) \mid bc$, in particular $H((\frac{k_1}{c}, \dots, \frac{k_n}{c})) \leq |bc|$. Therefore $|bc|^{4n^2} > H(x)$, and hence:
	    
	    \begin{equation}
	    \label{bh_bound_h_x}
	        (bh)^{4n^2} > H(x)
	    \end{equation}
	    
	    By Lemma \ref{H_fractions}, because $x$ has positive integer components, $H(x)$ is the maximum of these components. So (\ref{bh_bound_h_x}) tells us that no component in $x$ exceeds $(bh)^{4n^2}$.
	    
	    So if $(x, c) \in S_h$, then no component in $x$ can exceed $(bh)^{4n^2}$. Let $N$ be the number of such $x$; then by counting:
	    
	    \begin{align*}
	    N \leq& \prod_{i=1}^n \left\lceil \log_{a_i} (bh)^{4n^2} \right\rceil \\
	    =& \prod_{i=1}^n \left\lceil 4n^2 \log_{a_i} (bh) \right\rceil \\
	    \leq& \prod_{i=1}^n \left(4n^2 \log_{a_i} (bh) + 1 \right) \\
	    \intertext{and, because $h > \max a_i$:}
	    \leq& \prod_{i=1}^n \left(4n^2 \log_{a_i} (bh) + \log_{a_i} h \right) \\
	    \leq& \prod_{i=1}^n \left(5n^2 \log_{a_i} (bh) \right) \\
	    \leq& \prod_{i=1}^n \left(5n^2 \log_2 (bh) \right)\\
	    =& (5n^2 \log_2 (bh))^n.
	    \end{align*}
	    
	    Now note that by (\ref{defining_equation}), $c$ may be uniquely determined by $x$, so it follows that $|S_h| \leq (5n^2 \log_2 (bh))^n$. Therefore $|\pi(S_h)| \leq (5n^2 \log_2 (bh))^n$ as well. Thus:
	    
	    \begin{align*}
	    & \limsup_{h \to \infty} \frac{|\pi(S_2) \cap \{-h, \dots, h\}|}{2h+1} \\
	    =& \limsup_{h \to \infty} \frac{|\pi(S_2) \cap \{-h, \dots, h\}| - (2m+1)}{2h+1} \\
	    \leq& \limsup_{h \to \infty} \frac{|\pi(S_2) \cap \{-h, \dots, h\} \setminus \{-m, \dots, m\}|}{2h+1} \\
	    =& \limsup_{h \to \infty} \frac{|\pi(S_h)|}{2h+1} \\
	    \leq& \limsup_{h \to \infty} \frac{(5n^2 \log_2 (bh))^n}{2h+1} \\
	    =& \; 0.
	    \end{align*}
	    
	    So the upper natural density of $\pi(S_2)$ is zero (obviously it cannot be negative).
	    
	    Therefore $C = \pi(S) = \pi(S_1) \cup \pi(S_2)$ also has zero natural density, concluding the proof.
	\end{proof}
	
	This is not difficult to generalize to the condition on Presburger images that is the goal of this section:
	
	\begin{thmB}
        Let $E = \bigcup_i a_i^\mathbb{N}$, where $a_i \in \mathbb{Z}_{>1}$ for $1 \leq i \leq n$, and let $f : \mathbb{Z}^M \to \mathbb{Z}$ be a Presburger-definable function. Then $f(E^M)$ has zero natural density.
    \end{thmB}
	\begin{proof}
	    By (iii) in Lemma \ref{density_props}, it will suffice to cover the image of $f$ by sets of zero natural density.
	    
	    By Theorem 4.1 of \cite{ZP}, any Presburger-definable $f$ is piecewise linear, in the sense that its domain $E^M$ can be partitioned into finitely many sets such that the restriction of $f$ to each may be written as an affine function with rational coefficients. So it suffices to show that the image of $E^M$ under an affine function with rational coefficients has zero natural density, i.e. without loss of generality we may assume $f$ is affine with rational coefficients. In fact, we may further assume that the coefficients are integers and that the constant term is zero, because multiplication or translation by a constant integer does not affect whether the natural density of a set is zero by (v) and (vi) of Lemma \ref{density_props}. Lastly, we may assume $f$ actually depends on every argument, as otherwise we may simply treat $f$ as having a lower-dimensional domain.
	    
	    So the only case we need worry about is the case where $f$ is a linear function (with zero constant term) with nonzero integer coefficients. We write $f(x) = k_1 x_1 + \dots + k_M x_M$ (here $x = (x_1, \dots, x_M)$). Let $(j_i)_i$ be a sequence of $M$ numbers each between $1$ and $n$, inclusive. We note that there are only finitely many such sequences ($n^M$ precisely). Then consider the subset $S$ of the domain of $f$ in which $x_i \in (a_{j_i})^\mathbb{N}$ for each $i$. We note that the domain of $f$, $E^M$, is the union of these subsets over each sequence $(j_i)_i$.
	    
	    So it suffices to show that $f(S) = \{k_1 x_1 + \dots + k_M x_M : x \in S\}$ has zero natural density; but this is Lemma \ref{diophantine_density}.
	\end{proof}
	
	\section{A Dichotomy for Certain Expansions of Presburger Arithmetic}
	
	In this section, we aim to adapt a result of \cite{FM}, in which the authors prove that an expansion of the real ordered group (or any o-minimal expansion thereof) by a set $E$ with sufficiently sparse images gives rise to a dichotomy in the sparseness of its definable sets. Here we will show a similar result for expansions of Presburger arithmetic on the integers. Let $E \subseteq \mathbb{Z}$.
	
	The approach in \cite{FM} hinges on defining two collections $\mathcal{T}_n$ and $\mathcal{S}_n$ of subsets of $\mathbb{R}^n$. Here we do the same, defining collections of subsets of $\mathbb{Z}^n$. Our eventual goal is to utilize a form of cell decomposition for Presburger arithmetic in order to build any set definable in $(\mathbb{Z}, <, +, E)$ out of simpler pieces. In essence, $\mathcal{T}_n$ will contain the sets ``definable'' from subsets of $E$ via a particular union-of-intersections expression (but we put ``definable'' in quotes because the union and intersection are arbitrary), while $\mathcal{S}_n$ will contain specifically those sets in $\mathcal{T}_n$ formed from a \textit{weak cell,} a generalization of the notion of a Presburger cell that we will also define.
	
	The collections $\mathcal{S}_n$ and $\mathcal{T}_n$, as in \cite{FM}, are defined to satisfy four simultaneous claims, from which Theorem C will follow:
	
	\begin{enumerate}
	    \item $\mathcal{T}_n$ is a Boolean algebra.
	    \item Every element of $\mathcal{T}_n$ is a finite union of elements of $\mathcal{S}_n$.
	    \item The projection of a set in $\mathcal{S}_{n+1}$ onto the first $n$ coordinates is a set in $\mathcal{T}_n$.
	    \item Suppose that $A \in \mathcal{S}_{1}$ does not contain arbitrarily long pieces of the same arithmetic progression. Then there exist $M$ and $h : \mathbb{Z}^{M} \to \mathbb{Z}$ Presburger such that $A \subseteq h(E^M)$.
	\end{enumerate}
	
	The majority of this section will thus contain definitions of $\mathcal{T}_n$ and $\mathcal{S}_n$ and use them to prove these four claims. In order to do so, we will first define some notation previously used in \cite{FM}. For $X \subseteq \mathbb{Z}^{m+n}$ and $u \in \mathbb{Z}^m$, we let $X_u$ denote the fiber $\{x \in \mathbb{Z}^n : (u, x) \in X\}$. We will also make use of the \textit{diamond product} of two sets of integers:
	
	\begin{defn}
	    Given $A \subseteq \mathbb{Z}^l \times \mathbb{Z}^n$ and $B \subseteq \mathbb{Z}^m \times \mathbb{Z}^n$, we let
	
	    $$A \diamond B = \{(x, y, z) \in \mathbb{Z}^{l+m+n} : (x, z) \in A \wedge (y, z) \in B\}.$$
	
	    Let $u \in \mathbb{Z}^{l+n}$ and $v \in \mathbb{Z}^{m+n}$ be such that their last $n$ coordinates agree; then $\langle u, v \rangle$ denotes the tuple $(u_1, \dots, u_l, v) \in \mathbb{Z}^{l+m+n}$.
	\end{defn}
	
	Finally, we define a notation for certain expansions of first-order structures on the integers:
	
	\begin{defn}
	    Given a subset $E \subseteq \mathbb{Z}$, the notation $(\mathbb{Z}, <, +, E)^\#$ denotes $(\mathbb{Z}, <, +, (S))$, where $S$ ranges over all subsets of $E^k$ for $k \in \mathbb{N}$.
	\end{defn}
	
	Now we define $\mathcal{T}_n$:
	
	\begin{defn}
		Given $A \subseteq \mathbb{Z}^n$, we say $A \in \mathcal{T}_n$ iff there exist $m \in \mathbb{N}$, $X \subseteq \mathbb{Z}^{m + n}$ Presburger, and an indexed family $(P_\alpha)_{\alpha \in I}$ of subsets of $E^m$ such that $A = \bigcup_{\alpha \in I} \bigcap_{u \in P_\alpha} X_u$.
	\end{defn}
	
	Every Presburger subset of $\mathbb{Z}^n$ is in $\mathcal{T}_n$. Moreover, given $m \in \mathbb{N}, P \subseteq E^m,$ and Presburger functions $f_1, \dots, f_m : \mathbb{Z}^n \to \mathbb{Z}$, the set
	
	$$\{x \in \mathbb{Z}^n : (f_1(x), \dots, f_m(x)) \in P\}$$
	
	is in $\mathcal{T}_n$, because it is the union of all sets of the form
	
	$$\{x \in \mathbb{Z}^n : f_1(x) = u_1 \wedge \dots \wedge f_m(x) = u_m\}$$
	
	for $(u_1, \dots, u_m) \in P$.
	
	\begin{clm}
	\label{t_n_boolean_algebra}
		$\mathcal{T}_n$ is a Boolean algebra.
	\end{clm}
	\begin{proof}
		Let $A \in \mathcal{T}_n$. Then there exist $m, X \subseteq \mathbb{Z}^{m + n}$ Presburger, and an indexed family $(P_\alpha)_{\alpha \in I}$ of subsets of $E^m$ such that $A = \bigcup_{\alpha \in I} \bigcap_{u \in P_\alpha} X_u$. So:
		
		$$\mathbb{Z}^n \setminus A = \bigcap_{\alpha \in I} \bigcup_{u \in P_\alpha} (\mathbb{Z}^n \setminus X)_u$$
		
		and we can distribute the intersection over the union, giving some $K \subseteq \prod_{\alpha \in I} P_\alpha$ such that
		
		$$\mathbb{Z}^n \setminus A = \bigcup_{\gamma \in K} \bigcap_{u \in \gamma(I)} (\mathbb{Z}^n \setminus X)_u.$$
		
		Therefore $\mathcal{T}_n$ is closed under complementation.
		
		Similarly, let $B \in \mathcal{T}_n$. Then there exist $l$, $Y \subseteq \mathbb{Z}^{l+n}$ Presburger, and a family $(Q_\beta)_{\beta \in J}$ of subsets of $E^l$ such that $B = \bigcup_{\beta \in J} \bigcap_{v \in Q_\beta} Y_v$. Then:
		
		$$A \cap B = \bigcup_{(\alpha, \beta) \in I \times J} \bigcap_{(u, v) \in P_\alpha \times Q_\beta} (X \diamond Y)_{(u, v)}.$$
		
		Therefore $\mathcal{T}_n$ is closed under intersection. Closure under union of course follows from de Morgan's laws.
	\end{proof}
	
	As mentioned above, we will need the concept of a \textit{weak cell} in order to construct $\mathcal{S}_n$. As \cite{FM} mentions, the notion of a cell used for cell decomposition is very strict, defined inductively. This works well when we are showing that definable sets can be broken into simple pieces, but the standard definition is too restrictive to use here. So instead we effectively allow a weak cell in $n+1$ dimensions to comprise any definable set in the first $n$ dimensions and to be defined by its relationship to Presburger functions in the $(n+1)$-st dimension:
	
	\begin{defn}
		A \textit{weak cell} in $\mathbb{Z}^{n+1}$ is a set of one of the following forms:
		
		\begin{enumerate}[(i)]
			\item $S \times \{t \in \mathbb{Z} : t \equiv k \pmod{N}\}$
			\item $\{(x, t) \in \mathbb{Z}^{n+1} : x \in S, f(x) \leq t, t \equiv k \pmod{N}\}$
			\item $\{(x, t) \in \mathbb{Z}^{n+1} : x \in S, t \leq g(x), t \equiv k \pmod{N}\}$
			\item $\{(x, t) \in \mathbb{Z}^{n+1} : x \in S, f(x) \leq t \leq g(x), t \equiv k \pmod{N}\}$
		\end{enumerate}
	
		where $S \subseteq \mathbb{Z}^n$ and $f, g : \mathbb{Z}^n \to \mathbb{Z}$ are Presburger and $k, N \in \mathbb{Z}$.
	\end{defn}
	
	\begin{lem}
	\label{weak_cell_product}
		Let $A, B \subseteq \mathbb{Z}^m \times \mathbb{Z}^{n+1}$ be weak cells. Then $A \diamond B$ is a weak cell in $\mathbb{Z}^{2m+n+1}$.
	\end{lem}
	\begin{proof}
		Let $f_A, g_A, f_B, g_B$ each be either a Presburger function $\mathbb{Z}^{m+n} \to \mathbb{Z}$ or one of the symbols $-\infty, \infty$. We say that $f_A(x) \leq t$ if either $f_A$ is a Presburger function and the relation holds as written, or else $f_A$ is $-\infty$ and $t$ is any integer. We interpret $f_B(x) \leq t, t \leq g_A(x), t \leq g_B(x)$ analogously. Then we may write:
		
		$$A = \{(x, t) \in \mathbb{Z}^{m+n+1} : x \in S_A, f_A(x) \leq t \leq g_A(x), t \equiv k_A \pmod{N_A}\}$$
		$$B = \{(x, t) \in \mathbb{Z}^{m+n+1} : x \in S_B, f_B(x) \leq t \leq g_B(x), t \equiv k_B \pmod{N_B}\}$$
		
		for $f_A, g_A, f_B, g_B$ as above, $S_A, S_B \subseteq \mathbb{Z}^{m+n}$ Presburger, and $k_A, N_A, k_B, N_B \in \mathbb{Z}$.
		
		Then we have that:
		
		$$A \diamond B = \{(\langle x, y \rangle, t) \in \mathbb{Z}^{2m+n+1} : \langle x, y \rangle \in S_A \diamond S_B, f_A(x) \leq t \leq g_A(x), f_B(y) \leq t \leq g_B(y),$$
		$$t \equiv k_A \pmod{N_A}, t \equiv k_B \pmod{N_B}\}$$
		
		We will now define a weak cell $C$:
		
		$$C = \{(x, t) \in \mathbb{Z}^{2m+n+1} : x \in S_C, f_C(x) \leq t \leq g_C(x), t \equiv k_C \pmod{N_C}\}$$
		
		Let $S_C = S_A \diamond S_B$. Let $f_A^\diamond$ map $\langle x, y \rangle$ to $f_A(x)$ (or if $f_A$ is a symbol, $f_A^\diamond$ is the same symbol), and let $g_A^\diamond, f_B^\diamond, g_B^\diamond$ be defined analogously. Then we may let:
		
		$$f_C = \left\{ \begin{array}{cc}
		    \max f_A^\diamond, f_B^\diamond & \text{if $f_A, f_B$ are both functions} \\
		    \infty & \text{if $f_A = \infty$ or $f_B = \infty$} \\
		    f_A^\diamond & \text{if $f_B = -\infty$} \\
		    f_B^\diamond & \text{if $f_A = -\infty$}
		    \end{array} \right.$$
		    
		and:
		
		$$g_C = \left\{ \begin{array}{cc}
		    \min g_A^\diamond, g_B^\diamond & \text{if $g_A, g_B$ are both functions} \\
		    -\infty & \text{if $g_A = -\infty$ or $g_B = -\infty$} \\
		    g_A^\diamond & \text{if $g_B = \infty$} \\
		    g_B^\diamond & \text{if $g_A = \infty$}.
		    \end{array} \right.$$
		
		Lastly, let $k_C, N_C$ be such that $t \equiv k_C \pmod{N_C} \iff t \equiv k_A \pmod{N_A} \wedge t \equiv k_B \pmod{N_B}$. (We assume the latter condition is not impossible; if it is, then $A \diamond B$ is empty, and the empty set is trivially a weak cell.) The existence of such $k_C, N_C$ is guaranteed by the Chinese remainder theorem. Then we need only observe that $C = A \diamond B$ by definition.
	\end{proof}
	
	Now that weak cells have been defined, we define $\mathcal{S}_n$ analogously to $\mathcal{T}_n$ but with the supposition that the Presburger set that gives rise to each member of $\mathcal{S}_n$ be a weak cell:
	
	\begin{defn}
		For $A \subseteq \mathbb{Z}^{n+1}$, $A \in \mathcal{S}_{n+1}$ iff there exist $m \in \mathbb{N}$, a weak cell $C \subseteq \mathbb{Z}^{m+n+1}$, and an indexed family $(P_\alpha)_{\alpha \in I}$ of subsets of $E^m$ such that $A = \bigcup_{\alpha \in I} \bigcap_{u \in P_\alpha} C_u$. (We let $\mathcal{S}_0 = \mathcal{P}(\mathbb{Z}^0)$.)
	\end{defn}
	
	We will refer to the sets in $\mathcal{S}_{n+1}$ arising from each of the four different types of weak cell as being of types (i) through (iv) accordingly.
	
	In \cite{FM}, the authors make use of a very technical lemma in order to show that sets in $\mathcal{T}_n$ can be decomposed into sets in $\mathcal{S}_n$. We will need this same result here, but with $\mathbb{R}$ replaced by $\mathbb{Z}$:
	
	\begin{lem}
	\label{technical_union}
		Let $C_1, \dots, C_{k+1} \subseteq \mathbb{Z}^m \times \mathbb{Z}^n$, and let $(P_\alpha)_{\alpha \in I}$ be a family of subsets of $\mathbb{Z}^m$. Then $\bigcup_{\alpha \in I} \bigcap_{u \in P_\alpha} (C_1 \cup \dots \cup C_{k+1})_u$ is equal to the union:
		
		$$\bigcup_{\alpha \in I} \bigcap_{u \in P_\alpha} (C_1 \cup \dots \cup C_k)_u$$
		$$ \cup \bigcup_{\alpha \in I} \bigcap_{u \in P_\alpha} (C_{k+1})_u$$
		$$ \cup \bigcup_{\alpha \in I} \bigcup_{P_{\alpha, 1}, P_{\alpha, 2}} \bigcap_{(v, w) \in P_{\alpha, 1} \times P_{\alpha, 2}} ((C_1 \diamond C_{k+1}) \cup \dots \cup (C_k \diamond C_{k+1}))_{(v,w)}$$
		
		where, for $\alpha \in I$, $\{P_{\alpha, 1}, P_{\alpha, 2}\}$ are the partitions of $P_\alpha$ into two sets.
	\end{lem}
	\begin{proof}
	    The proof is elementary. We note that this lemma is exactly Lemma 2 of \cite{FM}, with $\mathbb{R}$ replaced by $\mathbb{Z}$. But there is no reference to the algebraic or topological structure of either, and neither does the result depend on countability, so this substitution makes no difference to the proof.
	\end{proof}
	
	We can now establish:
	
	\begin{clm}
	\label{s_n_finite_union}
		Every element of $\mathcal{T}_n$ is a finite union of elements of $\mathcal{S}_n$.
	\end{clm}
	\begin{proof}
		We claim first that for all $k \geq 1$, if $m \in \mathbb{N}$, $C_1, \dots, C_k$ are weak cells in $\mathbb{Z}^{m+n+1}$, and $(P_\alpha)_{\alpha \in I}$ is a family of subsets of $E^m$, then $\bigcup_{\alpha \in I} \bigcap_{u \in P_\alpha} (C_1 \cup \dots \cup C_k)_u$ is a finite union of elements of $\mathcal{S}_{n+1}$. When $k = 1$, this is the definition of $\mathcal{S}_{n+1}$. When $k > 1$, we use Lemma \ref{technical_union} to transform the expression into a union of corresponding expressions with a smaller value of $k$, and Lemma \ref{weak_cell_product} gives that the sets $(C_i \diamond C_{k})$ required in the process are still weak cells. So by induction, the above claim holds for all $k \geq 1$.
		
		We must now show that a set in $\mathcal{T}_n$ may always be written in this form. Recall that by definition,  $A \in \mathcal{T}_n$ iff there exist $m \in \mathbb{N}$, $X \subseteq \mathbb{Z}^{m + n}$ Presburger, and an indexed family $(P_\alpha)_{\alpha \in I}$ of subsets of $E^m$ such that $A = \bigcup_{\alpha \in I} \bigcap_{u \in P_\alpha} X_u$. So it suffices to show that any Presburger set $X \subseteq \mathbb{Z}^{m + n}$ is a union of weak cells $C_1 \cup \dots \cup C_k$. For this, we utilize a theorem of Cluckers \cite{Cluckers}. There, the author defines the notion of a \textit{Presburger cell} and proves that the domain of any Presburger function can be partitioned into finitely many Presburger cells such that the function is linear on each cell. Presburger cells in Cluckers's paper are necessarily weak cells by our definition, so this implies that $X$ can be partitioned into weak cells, concluding the proof.
	\end{proof}
	
	Next, we aim to show the following. Combined with the previous claim, this will show that the collection of finite unions of sets in $\mathcal{S}_n$ is closed under projection:
	
	\begin{clm}
	\label{s_n_projection}
		The projection of a set in $\mathcal{S}_{n+1}$ onto the first $n$ coordinates is a set in $\mathcal{T}_n$.
	\end{clm}
	\begin{proof}
		Let $A \in \mathcal{S}_{n+1}$. Then there exist $m \in \mathbb{N}$, a weak cell $C \subseteq \mathbb{Z}^{m+n+1}$, and an indexed family $(P_\alpha)_{\alpha \in I}$ of subsets of $E^m$ such that $A = \bigcup_{\alpha \in I} \bigcap_{u \in P_\alpha} C_u$. In other words, $x' \in \mathbb{Z}^{n+1}$ is in $A$ if and only if $\exists \alpha \in I : \forall u \in P_\alpha : (u, x') \in C$. So $x \in \mathbb{Z}^n$ is in $\pi A$, the projection of $A$ onto its first $n$ coordinates, when $\exists y \in \mathbb{Z} : \exists \alpha \in I : \forall u \in P_\alpha : (u, x, y) \in C$.
		
		Our goal is to show that $\pi A$ is in $\mathcal{T}_n$. To do this, we need to show that $x \in \pi A$ if and only if $\exists \alpha \in I' : \forall u \in P'_\alpha : \phi(u, x)$ for some natural number $m'$, index set $I'$, indexed collection $(P'_\alpha)_{\alpha \in I'}$ of subsets of $E^{m'}$, and Presburger formula $\phi$ with arity $m' + n$.
		
		We split into cases based on the type of $C$ (hence of $A$).
		
		\begin{enumerate}[(i)]
			\item In this case, $C = B \times \{t \in \mathbb{Z} : t \equiv k \pmod{N}\}$ where $B \subseteq \mathbb{Z}^{m+n}$ is Presburger and $k, N$ are constant integers. So:
			
			\begin{align*}
			x \in \pi A \iff& \exists y \in \mathbb{Z} : \exists \alpha \in I : \forall u \in P_\alpha : (u, x, y) \in C \wedge y \equiv k \pmod{N}\\
			\iff& \exists y \in \mathbb{Z} :  y \equiv k \pmod{N} \wedge \exists \alpha \in I : \forall u \in P_\alpha : (u, x) \in B\\
			\intertext{We no longer need to quantify over $y$, because this modular congruence can always be satisfied:}\
			\iff& \exists \alpha \in I : \forall u \in P_\alpha : (u, x) \in B
			\end{align*}
			
			\noindent This condition is now in the required form.
			
			\item In this case, $C = \{(x, t) \in \mathbb{Z}^{n+1} : x \in B, f(x) \leq t, t \equiv k \pmod{N}\}$ where $B \subseteq \mathbb{Z}^{m+n}$ and $f : \mathbb{Z}^{m+n} \to \mathbb{Z}$ are Presburger and $k, N \in \mathbb{Z}$. So:
			
			\begin{align*}
			x \in \pi A \iff& \exists y \in \mathbb{Z} : \exists \alpha \in I : \forall u \in P_\alpha : (u, x, y) \in C \\
			\iff& \exists y \in \mathbb{Z} : \exists \alpha \in I : \forall u \in P_\alpha : (u, x) \in B \wedge f(u, x) \leq y \wedge y \equiv k \pmod{N} \\
			\iff& \exists \alpha \in I : \exists y \in \mathbb{Z} : \forall u \in P_\alpha : (u, x) \in B \wedge f(u, x) \leq y \wedge y \equiv k \pmod{N} \\
			\intertext{We claim that the condition $y \equiv k \pmod{N}$ is unnecessary. Note that if there is some $y$ with $\forall u \in P_\alpha : f(u, x) \leq y$, we can merely add $1$ to $y$ until it satisfies the modular equivalence condition, without affecting the order condition. Therefore, the above is equivalent to:}
			\iff& \exists \alpha \in I : \exists y \in \mathbb{Z} : \forall u \in P_\alpha : (u, x) \in B \wedge f(u, x) \leq y\\
			\iff& \exists \alpha \in I : [\forall u \in P_\alpha : (u, x) \in B] \wedge [f(u, x) \text{ is bounded above for } u \in P_\alpha]\\
			\intertext{Because the values of $f$ are integers, $f$ is bounded above iff it has a maximum value on the given domain:}
			\iff& \exists \alpha \in I : [\forall u \in P_\alpha : (u, x) \in B] \wedge [\exists v \in P_\alpha : \forall u \in P_\alpha : f(u, x) \leq f(v, x)]\\
			\iff& \exists \alpha \in I : \exists v \in P_\alpha : \forall u \in P_\alpha : (u, x) \in B \wedge f(u, x) \leq f(v, x)\\
			\intertext{Now, note that we may let $I' = \{(\alpha, v) : \alpha \in I \wedge v \in P_\alpha\}$ and let $P'_{(\alpha, v)} = P_\alpha \times \{v\}$:}
			\iff& \exists (\alpha, v) \in I' : \forall (u, v) \in P'_{(\alpha, v)} : (u, x) \in B \wedge f(u, x) \leq f(v, x)
			\end{align*}
			
			\noindent This condition is now in the required form.
			
			\item In this case, $C = \{(x, t) \in \mathbb{Z}^{n+1} : x \in B, t \leq g(x), t \equiv k \pmod{N}\}$ where $B \subseteq \mathbb{Z}^{m+n}$ and $g : \mathbb{Z}^{m+n} \to \mathbb{Z}$ are Presburger and $k, N \in \mathbb{Z}$. So:
			
			\begin{align*}
			x \in \pi A \iff& \exists y \in \mathbb{Z} : \exists \alpha \in I : \forall u \in P_\alpha : (u, x, y) \in C\\
			\iff& \exists y \in \mathbb{Z} : \exists \alpha \in I : \forall u \in P_\alpha : (u, x) \in B \wedge y \leq g(u, x) \wedge y \equiv k \pmod{N}\\
			\iff& \exists \alpha \in I : \exists y \in \mathbb{Z} : \forall u \in P_\alpha : (u, x) \in B \wedge y \leq g(u, x) \wedge y \equiv k \pmod{N}\\
			\intertext{As in (ii), the condition $y \equiv k \pmod{N}$ is unnecessary, this time because we can always \textit{decrease} $y$ until it is satisfied:}
			\iff& \exists \alpha \in I : \exists y \in \mathbb{Z} : \forall u \in P_\alpha : (u, x) \in B \wedge y \leq g(u, x)\\
			\iff& \exists \alpha \in I : [\forall u \in P_\alpha : (u, x) \in B] \wedge [g(u, x) \text{ is bounded below for } u \in P_\alpha]\\
			\intertext{Because the values of $f$ are integers, $f$ is bounded below iff it has a minimum value on the given domain:}
		    \iff& \exists \alpha \in I : [\forall u \in P_\alpha : (u, x) \in B] \wedge [\exists v \in P_\alpha : \forall u \in P_\alpha : g(u, x) \geq g(v, x)]\\
			\iff& \exists \alpha \in I : \exists v \in P_\alpha : \forall u \in P_\alpha : (u, x) \in B \wedge g(u, x) \geq g(v, x)
			\end{align*}
			
			\noindent We can now apply the same technique as in (ii) to put this condition in the required form.
			
			\item In this case, $C = \{(x, t) \in \mathbb{Z}^{n+1} : x \in B, f(x) \leq t \leq g(x), t \equiv k \pmod{N}\}$ where $B \subseteq \mathbb{Z}^{m+n}$ and $f, g : \mathbb{Z}^{m+n} \to \mathbb{Z}$ are Presburger and $k, N \in \mathbb{Z}$. So:
			
			\begin{align*}
			x \in \pi A \iff& \exists y \in \mathbb{Z} : \exists \alpha \in I : \forall u \in P_\alpha : (u, x, y) \in C\\
			\iff& \exists y \in \mathbb{Z} : \exists \alpha \in I : \forall u \in P_\alpha : (u, x) \in B\\
			&\wedge f(u, x) \leq y \wedge y \leq g(u, x) \wedge y \equiv k \pmod{N}\\
			\iff& \exists \alpha \in I : \exists y \in \mathbb{Z} : \forall u \in P_\alpha : (u, x) \in B\\
			&\wedge f(u, x) \leq y \wedge y \leq g(u, x) \wedge y \equiv k \pmod{N}.
			\end{align*}
			
			This is the most complex case. To fully analyze it, we introduce a predicate $\psi$:
			
			$$\psi_{k, N}(a, b) \iff \exists y \in \mathbb{Z} : a \leq y \wedge y \leq b \wedge y \equiv k \pmod{N}$$
			
			The condition $\psi$ is Presburger (given constant $k$ and $N$). Now, again, the above condition implies that $f$ is bounded above and that $g$ is bounded below on the given domain. Therefore, $f$ has a maximum and $g$ a minimum on the given domain. So the above condition is equivalent to:
			
			\begin{align*}
			x \in \pi A \iff& \exists \alpha \in I : \psi_{k, N}(\max_{u \in P_\alpha} f(u, x), \min_{u \in P_\alpha} g(u, x)) \\
			\iff& \exists \alpha \in I : \exists v, w \in P_\alpha : [\forall u \in P_\alpha : f(u, x) \leq f(v, x)] \\
			&\wedge [\forall u \in P_\alpha : g(u, x) \geq g(w, x)] \wedge [\psi_{k, N}(f(v, x), g(w, x))] \\
			\iff& \exists \alpha \in I : \exists v, w \in P_\alpha : \forall u \in P_\alpha : f(u, x) \leq f(v, x) \\
			&\wedge g(u, x) \geq g(w, x) \wedge \psi_{k, N}(f(v, x), g(w, x))\\
			\intertext{Now let $I' = \{(\alpha, v, w) : \alpha \in I \wedge v, w \in P_\alpha\}$ and let $P'_{\alpha, v, w} = P_\alpha \times \{v\} \times \{w\}$:}
			\iff& \exists (\alpha, v, w) \in I' : \forall (u, v, w) \in P'_{\alpha, v, w} : f(u, x) \leq f(v, x)\\
			&\wedge g(u, x) \geq g(w, x) \wedge \psi_{k, N}(f(v, x), g(w, x))
			\end{align*}
			
			\noindent This condition is now in the required form.
		\end{enumerate}
	\end{proof}
	
	The next claim is, in essence, the heart of the result, which establishes a dichotomy for $\mathcal{S}_1$, under which any sufficiently sparse set is contained in a Presburger image of our original expansion set $E$:

	\begin{clm}
	\label{s_1_dichotomy}
		Let $A \in \mathcal{S}_{1}$. Suppose that $A$ does not contain arbitrarily long pieces of the same arithmetic progression; in other words, suppose that given $N', k'$ there exists $\ell$ such that no $\ell$ consecutive terms of the sequence $(N'i+k')_i$ lie in $A$. Then there exist $M$ and $h : \mathbb{Z}^{M} \to \mathbb{Z}$ Presburger such that $A \subseteq h(E^M)$.
	\end{clm}
    \begin{proof}
        Let $A$ be as given. By definition there exist $m \in \mathbb{N}$, a weak cell $C \subseteq \mathbb{Z}^{m+1}$, and an indexed family $(P_\alpha)_{\alpha \in I}$ of subsets of $E^m$ such that $A = \bigcup_{\alpha \in I} \bigcap_{u \in P_\alpha} C_u$. Assume that for each $\alpha$ the set $\bigcap_{u \in P_\alpha} C_u$ is nonempty (otherwise we may remove it from $I$ without affecting $A$). Moreover assume $I$ is nonempty, as otherwise $A$ is empty and the claim trivially follows.
        
        If $A$ is of type (i), then $A$ can be written as $\{t \in \mathbb{Z} : t \equiv k \pmod{N}\}$. But then $A$ fails our additional assumption for $(N', k') = (N, k)$. So this is impossible. If $A$ is of type (ii), $A$ is a union of intersections of sets of the form $\{t \in \mathbb{Z} : x \leq t, t \equiv k \pmod{N}\}$ for the same $k, N$ and varying numbers $x$; again, this is impossible because $A$ will contain all sufficiently high elements of the arithmetic progression for $N' = N, k' = k$. The case where $A$ is of type (iii) is analogous.
        
        So $A$ must be of type (iv); hence:
        
        \begin{align*}
        A =& \bigcup_{\alpha \in I} \bigcap_{u \in P_\alpha} \{(x, t) \in \mathbb{Z}^{m+1} : x \in B, f(x) \leq t \leq g(x), t \equiv k \pmod{N}\}_u\\
        =& \bigcup_{\alpha \in I} \bigcap_{u \in P_\alpha} \{t \in \mathbb{Z} : u \in B, f(u) \leq t \leq g(u), t \equiv k \pmod{N}\}
        \end{align*}
        
        for some $B, f, g$ Presburger and $k, N \in \mathbb{Z}$. We may remove the $u \in B$ condition by assuming each $P_\alpha$ is a subset of $B$. Then this is equivalent to:
        
        $$A = \bigcup_{\alpha \in I} \{t \in \mathbb{Z} : \max_{u \in P_\alpha} f(u) \leq t \leq \min_{u \in P_\alpha} g(u), t \equiv k \pmod{N}\}.$$
        
        We denote the set $\{t \in \mathbb{Z} : \max_{u \in P_\alpha} f(u) \leq t \leq \min_{u \in P_\alpha} g(u), t \equiv k \pmod{N}\}$ in this cover by $A_\alpha$.
        
        Now note that there exists $\ell$ such that no $\ell$ consecutive terms of $(Ni+k)_i$ may lie in $A$, hence in any $A_\alpha$. Therefore $\min_{u \in P_\alpha} g(u) < \max_{u \in P_\alpha} f(u) + \ell N$. By the well-ordering principle, $\max_{u \in P_\alpha} f(u)$ is well-defined as a maximum i.e. is in the image $f(E^m)$. Therefore every element of $A$ is equal to an element of $f(E^m)$ plus a nonnegative integer less than $\ell N$.
        
        Let $M = m+1$, and let $e_0, \dots, e_{\ell N - 1}$ be $\ell N$ distinct elements of $E$. Define $h : \mathbb{Z}^M \to \mathbb{Z}$ as:
        
        $$h(x_1, \dots, x_{m+1}) = \left\{ \begin{array}{cc}
		    f(x_1, \dots, x_m) + 0 & \text{if } x_{m+1} = e_0 \\
		    f(x_1, \dots, x_m) + 1 & \text{if } x_{m+1} = e_1 \\
		    \vdots & \vdots \\
		    f(x_1, \dots, x_m) + \ell N - 1 & \text{if } x_{m+1} = e_{\ell N - 1} \\
		    0 & \text{otherwise.}
		    \end{array} \right.$$
		    
		(Note that $h$ is Presburger, because it is a composition of $f$ with piecewise linear functions.) Then $h$ has the required property; i.e. as proven above, every element of $A$ is in $h(E^M)$.
    \end{proof}
    
    Using this dichotomy, we are now at last able to prove Theorem C.
    
    \begin{thmC}
        Let $\mathcal{I}$ be any set-theoretic ideal on $\mathbb{Z}$. Let $E \subseteq \mathbb{Z}$ be such that, for every $M \in \mathbb{N}$ and $h : \mathbb{Z}^M \to \mathbb{Z}$ Presburger-definable, the image $h(E^M)$ is in $\mathcal{I}$. Then every subset of $\mathbb{Z}$ definable in $(\mathbb{Z}, <, +, E)^\#$ either contains arbitrarily long pieces of some arithmetic progression or lies in $\mathcal{I}$.
    \end{thmC}
    \begin{proof}
        Consider $(\mathbb{Z}, (Y))$ where $Y$ ranges over all elements of all $\mathcal{S}_k$ for $k \in \mathbb{N}$. We demonstrated earlier that preimages of subsets of $E^k$ under Presburger functions are in $\mathcal{T}_n$, so in particular every such subset $S$ is also in some $\mathcal{T}_n$ and thus by Claim \ref{s_n_finite_union} is a finite union of sets in $\mathcal{S}_n$. We also demonstrated earlier that Presburger sets are also in $\mathcal{T}_n$. Therefore, every set definable in $(\mathbb{Z}, <, +, E)^\#$ is also definable in $(\mathbb{Z}, (Y))$. By combining Claims \ref{t_n_boolean_algebra}, \ref{s_n_finite_union}, and \ref{s_n_projection}, we see that in fact any subset of $\mathbb{Z}^n$ definable in $(\mathbb{Z}, (Y))$ is a finite union of sets in $\mathcal{S}_n$. We deduce that every subset $S$ of $\mathbb{Z}$ definable in $(\mathbb{Z}, <, +, E)^\#$ is a finite union of elements of $\mathcal{S}_1$. Call these elements $S_1$ through $S_r$.
        
        Assume that such an $S$ does not contain arbitrarily long pieces of the same arithmetic progression. Then neither do $S_1$ through $S_r$; by Claim \ref{s_1_dichotomy}, each $S_i$ is a subset of the image of $E^M$ for some $M$ under some Presburger function $h$. Said images, by assumption, lie in $\mathcal{I}$; hence, so do any subsets thereof, and any finite unions of subsets thereof. So $S \in \mathcal{I}$.
    \end{proof}
    
\bibliographystyle{amsplain}
\bibliography{biblio}
\end{document}